\documentclass{amsart}
\usepackage{amsthm, amssymb, amscd, verbatim}
\def\GG{\mathcal{G}}
\def\AA{\mathcal{A}}

\def\CC{\mathbb{C}}
\def\PP{\mathcal{P}}
\def\VV{\mathcal{V}}

\title{Universal Minimal Flow in the Theory of
Topological Groupoids}
\author{Riccardo Re}
\address{Dipartimento di Matematica e Informatica, Universit\`a di Catania,
Viale A. Doria 6, 95125 - Catania (Italy)}
\email{riccardo@dmi.unict.it}

\author{Pietro Ursino}
\address{Dipartimento di Scienza e Alta Tecnologia, Universit\`a dell'Insubria,
Via Valleggio 11, 22100 - Como (Italy)}
\email{pietro.ursino@uninsubria.it}
\begin{document}
\maketitle
\newtheorem{thm}{Theorem}
\newtheorem{defin}{Definition}
\newtheorem{prop}{Proposition}
\newtheorem{lm}{Lemma}
\newtheorem{cor}{Corollary}

\theoremstyle{remark}
\newtheorem{oss}{\bf Remark}
\newtheorem{example}{Example}
\newtheorem{domanda}{Question}
\newtheorem{notation}{Notation}

\section{Introduction}
In this paper we investigate some connections between Topological Dynamics, the theory of $G$-Principal Bundles, and the theory of Locally Trivial Groupoids.

Topological Dynamics is the study of continuous actions of (Hausdorff) topological groups $G$ on (Hausdorff) compact spaces $X$ ($G$-flows).  For notations and results of this theory we mainly follow and refer to \cite{Aus}, \cite{usp} and \cite{KPT}.

Every topological group $G$ has some natural compactifications. They
can be described as the maximal ideal spaces of certain function algebras and using the particular structures of those spaces it is possible to grasp new information about the group itself.
In particular the greatest ambit $S(G)$ is the compactification corresponding
to the algebra of all right uniformly continuous bounded
functions on $G$. There is a natural action of $G$ on $S(G)$ which has a universal property with respect to all actions of $G$ on compact spaces.  There also exists the universal minimal compact $G$-space $M_G$, which is a minimal $G$-flow which can be homomorphically mapped onto any other minimal $G$-flow. The flow $M_G$ can be derived from $S(G)$, indeed it can be constructed as any minimal flow of $S(G)$, so $S(G)$ describes, in a sense, universally the dynamics of $G$, or, in other terms, it makes explicit some general features which describe how $G$ acts on a generic compact space. It is natural to expect that $S(G)$ reflects many of the geometrical and algebraic properties of $G$, however it usually is a too big object to study in itself. In particular the homotopy theory of the group $G$, being reflected in the theory of principal $G$-bundles on spheres and, more generally, on locally contractible spaces, requires to develop a relative version of the $S(G)$ construction with analogous universal properties, more precisely it requires to study fiber bundles with fibers isomorphic to $S(G)$ over spaces with suitable topological properties.

Now there is a natural correspondence, developed by C. Ehresmann in \cite{Ehr}, between $G$-principal bundles over locally simply connected bases and locally trivial groupoids, see section 3, Definition 4.
In order to exploit this correspondence,  we move from groups towards locally trivial groupoids and from group actions on compact spaces towards groupoid actions on spaces with proper maps onto a locally compact base.
This allows us to extend part of topological dynamic to groupoid actions. For example, as we said above, every topological group $G$ has a universal minimal flow $M_G$,
here we prove the existence of an analogous one for groupoids $\GG=(G_0,G_1)$, with locally compact unit space $G_0$ and which act on spaces $X$ with a proper map $\rho:X\to G_0$ (see Section 2 for the definition of a topological groupoid and Section 5 for the definition of a groupoid action).
This reveals a correspondence between the fixed point on compacta property (e.g. Extreme Amenability) and the existence of global sections in the theory of fiber bundles. Besides that, it sheds some light to a possible relationship between contractibility and extreme amenability for arcwise connected groups. In a sense it seems that extremely amenable phenomenon can be divided in two categories extremely disconnected groups which mostly can be treated using combinatorial arguments as Ramsey Theory \cite{KPT} and connected groups which can be treated using concentration of measure technique \cite{gromi} \cite{GioPe1}. Our ultimate goal, for future work on the subject, is to deal with this latter class of groups using techniques of algebraic topology. The present paper, besides the results it contains, can possibly show that this perspective is not unreasonable.

The paper is organized as follows.
In Section \ref{sec:prelim} we review the concept of topological groupoids.
In Section \ref{sec:ehresmann} we give a self-contained treatment of Ehresmann's theory relating locally trivial groupoids with principal bundles.
In Section \ref{UnivComp} we rephrase the notion of greatest $G$-ambit in the context of the theory of principal bundles by defining $S(\PP)$ as a compactification of a $G$-principal bundle $P$.
In Section \ref{GroAct} we show how the groupoid acts on $S(\PP)$, considered itself as a fiber bundle.
In Section \ref{UnivProp} we shows that $S(\PP)$ can play the same role that $S(G)$ plays in the usual theory, since it enjoys a similar universal property.
Using this result we can assert in Section \ref{MinFlow} the universality of $M(\PP)$, as minimal subflow of $S(\PP)$.

In the last section we deal with the extremely amenable case and make comments on a possible perspective for future work aiming to use concepts and methods of algebraic topology for dealing with a class of extremely amenable groups with suitable topological properties.

\section{Preliminaries and Definitions}\label{sec:prelim}

In the following we will consider groupoids in the context of topological spaces. The most concise definition of a groupoid is the following.
\begin{defin}
A groupoid is a small category G in which every morphism is invertible.
\end{defin}

In order to make concrete the above definition, it is common to introduce the following notations.
The set of objects, or units, of $\GG$ will
be denoted by
$$G_0 = Ob(\GG).$$
The set of morphisms, or arrows, of $\GG$ will be denoted by
$$G_1 = Mor(\GG).$$

We will denote by $s(g)$ the source, i.e. the domain (respectively $t(g)$ the target, i.e. the range) of the morphism g. We thus obtain functions
$$s,t: G_1\rightarrow G_0$$
The multiplication operator $m:(g,h)\rightarrow gh$ is defined on the set of composable pairs of arrows $G_1\times_{r,s} G_1:=\{(g, h)\mid s(g)=r(h)\}$
$$m: G_1\times_{r,s} G_1\rightarrow G_1$$

The inversion operation is a bijection $()^{-1}:g\rightarrow g^{-1}$ of $G_1$.
Denoting by $u(x)$ the unit map of the object $x\in G_0$, we obtain an inclusion of
$G_0$ into $G_1$. We see that a groupoid $\GG$ is completely determined by the spaces
$G_0$ and $G_1$ and the structural morphisms $s,\ t,\ m,\ u,\ ()^{-1}$.  The structural maps satisfy the following properties:
\begin{itemize}
\item[(i)] $t(gh) = t(g), s(gh) = s(h)$ for any pair $(g, h)\in G_1\times_{r,s} G_1$, and the partially defined
multiplication $m$ is associative.
\item[(ii)] $s(u(x)) = r(u(x)) = x,\ \forall x\in G_0,\ u(t(g))g = g$ and $gu(s(g)) = g,\ \forall g\in G_1$
and $u : G_0\rightarrow G_1$ is one-to-one.
\item[(iii)] $t(g^{-1}) = s(g), s(g^{-1}) = t(g), gg^{-1} = u(t(g))$ and $g^{-1}g = u(s(g))$.
\end{itemize}
Then one can define a topological groupoid as follows.
\begin{defin}
A groupoid $\GG$ is a topological groupoid if $G_0$ and $G_1$ are topological Hausdorff spaces and all the structure maps
$s,t,m,u,()^{-1}$ are continuous functions.
\end{defin}

\begin{notation}
For our convenience the multiplication $m$ of the elements $g,h\in G_1$ will follow inverse notation, that is
$s(gh)=s(g)$ and $t(gh)=t(h)$. For $x,y\in G_0$ we define $$G[x,y]=\{g\in G_1\ |\ s(g)=x,\ t(g)=y\}.$$ By abuse of notation we denote by $G[x]$ the set of morphisms $G[x,x]$. It is useful for the following to observe that $G[x]$ is actually a group.
\end{notation}
\begin{defin}
A groupoid $\GG$ is {\em transitive} if for all $x,y\in G_0$ there exists $g\in G_1$ such that $s(g)=x$ and $t(g)=y$.
\end{defin}
Some authors define such a groupoid as {\em connected} groupoid, we avoid this term to prevent any misleading with its topological analogue.
Henceforth we assume that a groupoid $\GG$ is always a transitive topological groupoid.
One gets easily the following consequence of this property.

\begin{prop} If $\GG$ is transitive then all the groups $G[x]$ are isomorphic.\end{prop}

\section{Ehresmann Groupoids}\label{sec:ehresmann}

We begin this section with a self-contained review of some results due to C. Ehresmann, see \cite{Ehr}, who established an equivalence between $G$-principal bundles and locally trivial groupoids, see Definition \ref{def:loctriv} below.
Henceforth we will assume all topological spaces to be Hausdorff spaces.
\begin{defin}\label{def:loctriv} Let $\GG$ be a groupoid.  $\GG$ is ``locally trivial" if for all $x\in G_0$ there exists an open set $U$ such that for any $x\in U$ there exists a continuous section $\tau :U\to G_1$ of the map $t$ such that $s\circ \tau(y)=x$ for all $y\in U$, or equivalently, such that $s\circ \tau$ is the constant map $x$. \end{defin}
\begin{oss} In the above definition, if one requires $\tau(y)\in G[x,y]$ for all $y\in U$ one obtains an equivalent definition. One can also exchange the roles of $s$ and $t$ obtaining again an equivalent definition.
\end{oss}

\begin{defin}\label{def:fibprin} A principal bundle on a topological space $X$, with structural topological group
$G$, is a map $p:P\to X$ such that \begin{enumerate}
\item There exists a right action on
$P\times G\to P$ which preserves fibers of
$p$ and $G$ freely and transitively acts on $p^{-1}(x)$, for all $x\in X$.
\item The map $p$ identifies $X$ with $P/G$, the quotient space.
\end{enumerate}\end{defin}

\begin{example} The simplest example  of principal bundle on the space $X$ is the trivial bundle
$X\times G$ with $p$ the projection on the first coordinate. If the space $X$ is contractible then all the principal bundles on $X$ are trivial.
\end{example}

One has the following application of the example above.

\begin{prop}\label{prop:loctrivial} If $X$ is locally contractible (namely, all
$x\in X$ has a contractible neighborhood), then all bundles $p: P\to X$ are locally trivial. More precisely for all
$x\in X$ there exists a neighborhood $U\subset X$ and $x\in U$ such that $p|_{p^{-1}(U)}: p^{-1}(U)\to U$ is a trivial $G$-bundle.
In particular there exists a $G$-invariant homeomorphism $p^{-1}(U)\cong U\times G$. \end{prop}

\begin{example}
Let $P=\CC^{n+1}\setminus\{0\}$, $X=\CC\mathbb{P}^n$ be the $n$-dimensional complex projective space, let $G$ be $\CC^\ast$, the multiplicative group on the complex space, with the topology induced by $\CC$, and $p:P\to X$, the canonical projection $v\mapsto [v]$, where $[v]$ is the equivalence class of the relationship $v\sim \lambda v$ for $\lambda\in \CC^\ast$.
This is an example of non trivial principal bundle with structural group
$G=\CC^\ast$. The action of
$G$ on $P$ is the multiplication $v\mapsto \lambda v$, for $v\in \CC^{n+1}\setminus\{0\}$ and $\lambda\in \CC^\ast$. \footnote{In this case it doesn't matter which action we decide to assume, right or left, since the group is commutative.}
If $n\geq 1$ this is not a trivial bundle.

\end{example}

From now on we assume that $X$ is a locally contractible space and hence all bundles on $X$ are locally trivial.

\begin{defin}\label{groupoid}For any principal $G$-bundle $\pi:P\to X$ one defines the groupoid $\GG(P)$ as follows. Let
$G_0=X$ and for any $x,y\in G_0$ $$\GG(P)[x,y]=\{f:\pi^{-1}(x)\to \pi^{-1}(y)\ |\ f\mbox{ equivariant}\}.$$
We say that $f$ is equivariant whenever
$f(zg)=f(z)g$.
Equivalently, denoting by $G_1$ the space of arrows of $\GG(P)$, we can define $G_1$ as the set of equivalence classes $(P\times P)/\sim$
under the equivalence relation where $(u,v)\sim(u',v')$ iff there is $g\in G$ such that $u'=ug$ and $v'=vg$.
In other terms $G_1$ is the set of the orbits of the diagonal action of $G$ on $P\times P$, defined by
$(u,v)g=(ug,vg)$. \end{defin}
\begin{notation} We denote by
$[u,v]$ the orbit of $(u,v)$ under the action of $G$. \end{notation}
In the following proposition we describe the structural maps of
$\GG(P)$ as a groupoid.

\begin{prop}\label{prop:GP} $\GG(P)=(G_1,G_0)$ is a groupoid with the following properties.
\begin{enumerate} \item The maps $s,t: G_1\to G_0$ are respectively defined by $s([u,v])=p(u)$ and $t([u,v])=p(v)$, where $p:P\to X=G_0$ is the bundle projection map.
\item The map $u:G_0\to G_1$ is defined by $u(x)=[v,v]$, with $v\in p^{-1}(x)\subset P$ arbitrary.

\item The composition law $G_1\times_{t,s}G_1\to G_1$ is defined by $[u,v][v',w']=[ug,w']$, iff there exists $g\in G$ such that $v'=vg$.
\item The inverse map $G_1\to G_1$ defined by $[u,v]\mapsto [v,u]$.
\item For any fixed $u_0\in P$, let $x_0=p(u_0)$, then there exists  a unique isomorphism $\phi_{u_0}: G[x_0]=\{[u,v]\in G_1\ |\ p(u)=p(v)=x_0\}\to G$ defined by $\phi_{u_0}([w,u_0])= g$ iff $w=u_0g$. For any other choice $u_1=u_0h\in p^{-1}(x_0)$ one has $\phi_{u_1}([u,v])=h^{-1}\phi_{u_0}([u,v])h$ for any $[u,v]\in G[x]$.
\item For any fixed $u_0\in P$, let $x_0=p(u_0)$, then there exists  a unique isomorphism $\psi_{u_0}: G[x_0]=\{[u,v]\in G_1\ |\ p(u)=p(v)=x_0\}\to G$ defined by $\psi_{u_0}([u_0,w])=g^{-1}$ iff $w=u_0g$. For any other choice $u_1=u_0h\in p^{-1}(x_0)$ one has $\psi_{u_1}([u,v])=h^{-1}\psi_{u_0}([u,v])h$ for any $[u,v]\in G[x]$.

\item Fixed $x_0, u_0$ as above, there exists a unique bijective map $\tau_{u_0}:t^{-1}(x_0)\to P$ defined by $\sigma_{u_0}([w,u_0])= w$.
If we interchange $u_0$ with $u_1=u_0h$ then one has $\tau_{u_1}(y)=\tau_{u_0}(y)h$ for any $y\in t^{-1}(x_0)$.
By identifying $G[x_0]=G$ through $\psi_{u_0}$, the map $\tau_{u_0}$ becomes a right $G$-invariant map.
Besides the following holds $p\circ \tau_{u_0}=s$.
\item Analogously, there exists a unique bijective map
$\sigma_{u_0}:s^{-1}(x_0)\to P$ defined by $\sigma_{u_0}([u_0,w])= w$ which, by identification $G[x_0]=G$ through the isomorphism $\phi_{u_0}$, is $G$-invariant, and one obtains $p\circ\sigma_{u_0}=t$.
\end{enumerate}

\end{prop}
\begin{proof}

The maps $s,t$  are well defined since $p(zg)=p(z)$ for any $z\in P$ and $g\in G$.
The map $u:G_0\to G_1$ is well defined as well, since a different choice of $v'\in p^{-1}(x)$ yields $[v',v']=[v,v]$.
In order to prove that the product is well defined, let us consider $[u,v]=[uh,vh]$,
therefore, for $v'=vg$, we get $[uh,vh][v',w]=[uh,vh][(vh)h^{-1}g,w]=[uh,wg^{-1}h]=[u,wg^{-1}]$. In a similar way if $[v'',w'']=[v',w]$ then $v''=v'h=vgh$ e $w'=wh$, which in turns implies $[u,v][v'',w']=[u, w'h^{-1}g^{-1}]=[u,wg^{-1}]$.
The inverse map is, obviously, well defined.
The properties to be a groupoid are easily verified. For example, since the product is well defined, we can deduce associativity from the following fact
$$([u,v][v,w])[w,z]=[u,z]=[u,v]([v,w][w,z]).$$

\paragraph{(5)} The map $\phi_{u_0}$ is bijective and well defined, since it is the unique representation of $[u,v]\in G[x]$ as $[u,v]=[w,u_0]=[u_0g,u_0]$.
Moreover $\phi_{u_0}$ is a group homomorphism. Indeed, one has
$$\phi_{u_0}([u_0g,u_0][u_0h,u_0])=\psi_{u_0}([u_0gh,u_0])=gh=\phi_{u_0}([u_0g,u_0])\phi_{u_0}([u_0h,u_0]).$$
Since $\phi_{u_1}([u,v])=\phi_{u_1}([wh,u_0h])=\phi_{u_1}([u_0gh,u_1])=\phi_{u_1}([u_1h^{-1}gh,u_1])=h^{-1}gh=h^{-1}\phi_{u_0}([u,v])h$.
\paragraph{(6)} Similar to the proof of (5).
\paragraph{(7)} The map $\tau_{u_0}: t^{-1}(x_0)\to P$ is bijective and well defined. We observe that again by replacing $u_0$ with $u_1=u_0h$ we get $\tau_{u_1}([w,u_0])=\tau_{u_1}([wh,u_1])=wh=\tau_{u_0}([w,u_0])h$.  Let $[u,v]\in G[x]$,
we may assume $[u,v]=[u_0,u_0g]$, therefore we can identify this element with $g^{-1}\in G$ trough $\psi_{u_0}$. We thus obtain \begin{eqnarray*}\tau_{u_0}([w,u_0]\cdot g^{-1})=\tau_{u_0}([w,u_0][u_0,u_0g])=\tau_{u_0}([w,u_0g])\\
=\tau_{u_0}[wg^{-1},u_0]=wg^{-1}=\tau_{u_0}([w,u_0]) g^{-1}.\end{eqnarray*}  Finally $p(\tau_{u_0}([w,u_0]))=p(w)=s([w,u_0])$.
\paragraph{(8)} Analogous to the proof of (7).
\end{proof}

\begin{oss} Let us assume $X=G_0=\{1\}$ and $P=G\to X$, with $p$ equal to the constant map $G\to\{1\}$. Then
$\GG(P)=G\times G/\sim$ is the collection of the orbits $[a,b]$ with $a,b\in G$, by choosing as canonical representants $[ab^{-1},1]$.
If we proceed in analogy with the construction of Proposition \ref{prop:GP} we get  $\GG(P)\cong G$, using the bijection $[g,1]\mapsto g$.
 \end{oss}
\begin{proof}
In order to verify that the map $[g,1]\mapsto g$ is an isomorphism, it is sufficient to consider the following calculation $[g,1][h,1]=[g,1][1,h^{-1}]=[g,h^{-1}]=[gh,1]$.\end{proof}

Let $G_0$ be a connected and locally simply connected Hausdorff space. Under the above assumptions the following holds.

\begin{thm}[Ehresmann]\label{thm:ehres1} Let $P\to X$  be a $G$-principal bundle then there exists a natural topology on $\GG(P)$ induced by $P$ such that $\GG(P)$
is a locally trivial topological groupoid.
\end{thm}
\begin{proof}
In order to define in a meaningful way a topology on the quotient $(P\times P)/ \sim$ we use the second version of the Definition \ref{groupoid}.

Assuming $\pi:P\times P\to \GG(P)$ is the canonical projection, a set $\pi^{-1}(\AA)$ is an open set iff, given an arbitrary $(u,v)\in \pi^{-1}(\AA)$, there exist two open neighborhoods $U,V$ respectively of $u,v\in P$, such that
$U\times V\subset \pi^{-1}(\AA)$,
 which implies $\pi(U\times V)\subset \AA$, and in particular $\pi^{-1}(\pi(U\times V))\subset \pi^{-1}(\AA)$.

  Since $\pi^{-1}(\pi(U\times V))=\{(u'g,v'g)\ |\ (u',v')\in U\times V,\ g\in G\}=\bigcup_{g\in G}(U\times V)g$, $\pi^{-1}(\pi(U\times V))$ is an open set.
  Hence the collection of sets as $\pi(U\times V)$ is a basis of $\GG(P)$.
  Observe that, since $U$ and $V$ can be chosen in an arbitrary basis of open sets, by Proposition \ref{prop:loctrivial} they, in particular, can be of the following type $U\cong U_0\times W_1$ and $V\cong V_0\times W_1$, with $U_0$ and $V_0$ contractible neighborhood of $x=\pi(u)$ and $y=\pi(v)$ respectively, and $W_1$ in a basis of the Identity $1\in G$.
  Now we are ready to construct the following sections
  $U_0\to U\cong U\times W_1$ and $V_0\to V\cong U\times W_1$, both fixed by the map $x\mapsto (x,1)$ created using the inverses of the above isomorphisms.
  \vskip2mm
  \begin{notation} We denote by $U'_0, V'_0$
  the images of $U_0$ and $V_0$ inside $U$ and $V$ through the above sections, and $x'$, $y'$ the images of elements $x\in U_0$ and $y\in V_0$, respectively.
  \end{notation}

  \paragraph{} Note that we can identify $U=U'_0W_1$ and $V=V'_0W_1$.
Then \begin{eqnarray*}U\times V&=&\{(x'g,y'h)\ |\ x'\in U'_0,\ y'\in V'_0,\ g,h\in W_1\},\\
\pi^{-1}(\pi(U\times V))&=&\{(x'gk,y'hk)\ |\ x'\in U_0',\ y'\in V'_0,\ g,h\in W_1,k\in G\}.\end{eqnarray*}

 Since $[x'gk,y'hk]=[x'gh^{-1},y']$, we can find the bijection $U_0\times V_0\times W_1W_1^{-1}\cong \pi(U\times V)$, which can be defined explicitly as follows.
\begin{equation}\label{eq:psi}
\psi: U_0\times V_0\times W_1W_1^{-1}\to \pi(U\times V),\quad\psi(x,y,r)=[x'r,y'].\end{equation}

 The map $\psi$ is clearly surjective. Moreover if
 $[x'r,y']=[x''s,y'']$ then $y'=y''g$ and $x'r=x''sg$ for some $g\in G$, but, since $y',y''\in V_0'$, with $V'_0$ local section of $p:P\to X$ onto $V_0$,
 we deduce $g=1$, hence $y'=y''$ and $x'r=x''s$, therefore $x'=x''sr^{-1}$. Since $U'_0$ is a section of $p:P\to X$ onto $U_0$, this yields $s=r$, hence $\psi$ is injective.
 Choosing $U,V$ as above then we obtain a concrete description of an open basis of $\GG(P)$.

Finally we fix arbitrarily
$x\in X=G_0$ and $[u,u]\in\GG(P)$ with $p(u)=x$.
In this case for the above construction we can choose
$U=V\cong U_0\times W_1$.

 Therefore we can define the section $\sigma:U_0\to U_0\times U_0\times W_1W_1^{-1}\cong \pi(U\times U)$  of the target map $t$ through $x\mapsto (x,x,1)$ on the open set $U_0\subset G_0$. This section has the properties requested by Definition \ref{def:loctriv}.
The continuity properties of the structural maps are obtained by routine arguments, then one shows that
$\GG(P)$ is a topological groupoid.
\end{proof}

Conversely, let
$\GG$ be a locally trivial topological groupoid, then the following holds.
\begin{thm}[Ehresmann]\label{thm:fibratogpd} Let $\GG$ a locally trivial groupoid with $G_0$ locally contractible, then for any $x\in G_0$ the source map $s:P=t^{-1}(x)\to G_0$
is a right principal bundle with structural group $G=G[x]$ and $\GG\cong\GG(P)$.

Similarly, the target map
$t:P=s^{-1}(x)\to G_0$ is a left principal bundle with structural group $G=G[x]$. \end{thm}

\begin{proof} It follows from (6) of Proposition \ref{prop:GP}.\end{proof}

\section{Universal compactifications of principal bundles}\label{UnivComp}

We begin with a brief review of a compactification of topological groups $G$ endowed with a universal property with respect to actions of $G$ on compact spaces, widely studied in the theory of Universal Minimal Flows, see \cite{Aus} and \cite{usp} for a useful survey. Then we will explain how to generalize the construction to principal bundles.
\begin{notation}
Let $G$ be a topological group, we recall that the $C^\ast$-algebra $R_G=RUC^b(G)$ of {\em right uniformly continuous} functions, is defined as the algebra of  bounded continuous functions that satisfy the following $$\forall\ \epsilon>0\ \exists\ \VV:\ \forall\ y\in G\ \mbox{ si ha }|f(y)-f(gy)|<\epsilon,\ \forall\ g\in \VV,$$
where the set $\VV$ is taken in a system of open neighborhoods of the identity element $1\in G$.\end{notation}
Under pointwise addition, multiplication and conjugation, and with the sup norm
$R_G$ is an abelian $C^\ast$-algebra.
The Gelfand duality associates to $R_G$ a compact space, $S(G)$, widely investigated in \cite{usp} and \cite{KPT}.
Symmetrically we can define, by substituting the inequality $|f(y)-f(gy)|<\epsilon$ with $|f(y)-f(yg)|<\epsilon$,
$L_G=LUC^b(G)$, the algebra of {\em left uniformly continuous}.

By a natural isomorphism between $R_G$ and $L_G$, the latter associates to the same compact space $S(G)$.
In the articles \cite{usp} and \cite{KPT} the authors need to assume $G$ metrizable, with right invariant compatible metric, in order to separate the elements of $G$. In the following we will consider $G$ with the same properties and $\GG(P)$ the topological groupoid associated to a $G$-principal bundle as defined in Section \ref{sec:ehresmann} with $X$ locally compact and locally contractible.
We will need to assume $X$ locally compact in order to be able to exploit the Gelfand construction.
Hence we will be able to refer to the canonical locally compact space associated to the $C^\ast$-algebra $C_0(X)$, the algebra of continuous functions $f:X\to \CC$, endowed with sup-norm and compact support. Indeed, the topology of a locally compact space $X$ is induced by $C_0(X)$ in the sense that closed sets are the ideals of $C_0(X)$.

Recall that Gelfand construction provides a categorial duality between the category of compact spaces and the category of unital commutative $C^\ast$-algebras (see \cite{murphy}, Theorems 2.1.10 and 2.1.15).
In the sequel we will need the one directional functorial version of Gelfand construction dealing with locally compact Hausdorff spaces instead of compact ones. Indeed,
the space of characters $\Omega(A)$ of a commutative $C^\ast$-algebra $A$ is locally compact, $A=C_0(\Omega(A))$, see \cite{murphy} Theorems 1.3.5 e 2.1.10, and an homomorphism
$\phi^\ast:A\to B$ associates to a continuous map $\phi:\Omega(B)\to \Omega(A)$.
In particular the following holds:
\begin{prop}\label{prop:semigelf} Let $A$ and $B$ be two isomorphic commutative $C^\ast$-algebras then the related locally compact spaces $\Omega(A)$ and $\Omega(B)$ are homeomorphic.
\end{prop}
Observe that the existence of a continuous map $f:X\to Y$ between locally compact spaces does not necessary implies the existence of an homomorphism $f^\ast:C_0(Y)\to C_0(X)$, since this happens only when the map $f$ is proper.

\begin{notation} Fix $x_0\in X$ and $u_0\in P$ such that $p(u_0)=x_0$. Let us consider the {\em left} $G$-principal bundle $P=s^{-1}(x_0)$, already introduced by
 Theorem \ref{thm:fibratogpd}.\end{notation}

 Recall that in this case the canonical projection $p:P\to X$ is identified with the map $t:s^{-1}(x_0)\to X$.
 Let us notice that $P$ can be described as the set $s^{-1}(x_0)=\{[u_0,y']\ |\ y'\in P\}$ and for every element $\tau\in s^{-1}(x_0)$ the representation $\tau=[u_0,y']$ is unique.
 In this representation, the map $t:s^{-1}(x_0)\to X$ is clearly defined by $t([u_0,y']=p(y')$.

Next definition generalizes $RUC_0^b(G)$ to the bundle $P$.

\begin{defin}\label{def:rucbGP} Denote $R_{\PP,u_0}$ the set of functions $f:s^{-1}(x_0)\to\CC$ with the following properties
\begin{enumerate}
\item $f$ bounded, continuous and $f= 0$ outside $t^{-1}(K)$, where $K$ is a compact subset of $X$.
\item For all local section $x\mapsto x'$ of $p:P\to X$ defined on the open set $U\subset X$ and for every $\epsilon>0$ there exists an open neighbourhood $\VV$ of $1\in G$ such that for every $g\in \VV$, for every $h\in G$ and for every $x\in U$ we have $|f([u_0gh,x'])-f([u_0h,x'])|<\epsilon$.
\end{enumerate}
\end{defin}
\begin{oss}In the above definition the open set $\VV$ is independent from the choice of the section $x\to x'$.
\end{oss}
\begin{proof}Let us assume $x\mapsto x''$ is a different section of $p$ on $U$, then one can write $x'$ as $x''k(x)$, with $k:U\to G$ selected continuous function.
 Hence $[u_0h,x'']=[u_0h, x'k^{-1}(x)]=[u_0hk(x),x']$, then a function $f$, such as one of the above definition, is such that $|f([u_0ghk(x),x'])-f([u_0hk(x),x'])|<\epsilon$, for any $x\in U$ and $g\in\VV$, by the arbitrariness of $h$ in the Definition \ref{def:rucbGP}. Therefore for any $g\in\VV$ we have
$|f([u_0gh,x'')-f([u_0h,x''])|<\epsilon$. \end{proof}
We will use the above argument to replace in Definition \ref{def:rucbGP} the sentence ``for all sections" with ``there exists a section"
This modification has fruitful consequences, indeed since there exists a local section $p:P\to X$ one can introduce a more classical local notation for functions $f\in R_{\PP,u_0}$.

\begin{prop}\label{prop:RPloc} Let $U\subset X$ be an open set and assume there exists on $U$ a section $x\mapsto x'$ of the projection $p:P\to X$.
As above we identify $P\cong s^{-1}(x_0)$ as a left $G$-bundle.
Then the map $\psi': U\times G\to p^{-1}(U)$ defined by $(x,h)\mapsto [u_0h,x']$ is an homeomorphism.
In particular the restrictions of all functions $f:P\to \CC$ to $p^{-1}(U)$ can be identified with the functions $F:U\times G\to \CC$ by putting $F(x,h)=f([u_0h,x'])$.
Moreover the condition (2) of Definition \ref{def:rucbGP} can be rephrased using the following inequality $|F(x,gh)-F(x,h)|<\epsilon$.\end{prop}
\begin{proof} The map $\psi'$ is a restriction of the homeomorphism $\psi$ defined by (\ref{eq:psi}) in the proof of Theorem \ref{thm:ehres1}.
This in turns implies that
$\psi'$ is an homeomorphism and, by a straight calculation, the characterization of the condition (2) of Definition \ref{def:rucbGP}. \end{proof}
Using the following lemma we can construct many useful examples of functions in $R_{\PP,u_0}$.

\begin{lm}\label{lm:RG}
 Let $\eta: G_0\to \CC$ be a continuous function with compact support $H\subset U$, with $U$ open set of $X$ on which there exists a section of $p:P\to X$.
 Moreover let $\theta:G\to\CC$ be a function in $R_G$, then, using notations of Proposition \ref{prop:RPloc}, the function $F:U\times G\to \CC$ defined by $F(x,h)=\eta(x)\theta(h)$ belongs to $R_{\PP,u_0}$.
\end{lm}
\begin{proof} The property (1) of the Definition \ref{def:rucbGP} is verified by using the hypothesis on the support of $\eta$. Instead property (2) follows by inequality $$|F(x,gh)-F(x,h)|\leq (\max|\eta|)|\theta(gh)-\theta(h)|$$ and the fact that $\theta \in R_G$.
\end{proof}
\begin{defin} Denote by $S(\PP,u_0)=\Omega(R_{\PP,u_0})$, the locally compact space corresponding, by Gelfand construction, to the commutative $C^\ast$-algebra $R_{\PP,u_0}$.
Once $u_0\in P$ is fixed, we will denote $R_{\mathcal{P},u_0}=R_\PP$ e $S(\mathcal{P},u_0)=S(\PP)$.
 \end{defin}
 Recall that the space $S(\PP)=\Omega(R_\PP)$ is composed of surjective ring homomorphisms $\xi:R_\PP\to \CC$ and that there exists a natural map
$i: P\to S(\PP)$ defined as follows: $i(\tau)=\hat{\tau}$, where $\hat{\tau}\in S(\PP)$ is the homeomorphism $\hat{\tau}:R_\PP\to \CC$ such that $\hat{\tau}(f)=f(\tau)$.

\begin{thm} $i: P\to S(\PP)$ is a continuous injective map with dense image.
\end{thm}
\begin{proof} Since the open sets of $S(\PP)$ have as pre-basis sets like $f^{-1}(\Omega)$, with $f\in R_\PP$ and $\Omega\subset \CC$ open set, whose counterimages by $i$ are open sets of $P$, the continuity of $i$ easily follows.

Moreover, a continuous function $f\in R_\PP$  which is zero in $i(P$) is equal to zero since $f(\hat{\tau})=\hat{\tau}(f)=f(\tau)=0$ for any $\tau\in P$. Therefore the subspace
$i(P)\subset S(\PP)$ is dense.

  Let$[u_0,v]$ and $[u_0,w]$ a pair of distinct elements of $P$. If $p(v)\not=p(w)$ then, since $X$ is locally compact, we can find a function $\eta\in R_\PP$ like
   $\eta(x)$, exactly as in Lemma \ref{lm:RG}, such that $\eta(p(v))\not=\eta(p(w))$, hence $i([u_0,v])\not=i([u_0,w])$.

   If $p(v)=p(w)=y$, then $[u_0,w]=[u_0,vh^{-1}]=[u_0h,v]$ for some $h\in G$ with $h\not=1$.
Consider a function $\theta:G\to \CC$ in $R_G$ such that $\theta(1)\not=\theta(h)$.
It follows that there exists a local section $x\mapsto x'$ defined in a neighborhood $U$ of $y\in X$ such that $y'=v$.
Let $\eta$ be a function with compact support inside $U$ such that $\eta(y)=1$.
Using notations of Proposition \ref{prop:RPloc}, the function $G:U\times G\to \CC$ built in Lemma \ref{lm:RG} is such that $F(y,h)=\theta(h)\not=\theta(1)=F(y,1)$.
Therefore $i([u_0,v])\not=i([u_0,v'])$.
\end{proof}

   \section{Groupoid Actions}\label{GroAct}
   We recall the definition of right action of a topological groupoid on a topological space.

\begin{defin}\label{def:azionegpd} Let us assume that there exists continuous maps
\begin{itemize}

  \item the {\bf anchor map} $\rho: Y\to G_0,$
  \item $\alpha: Y\times_{\rho,s}G_1\to Y,$, where $Y\times_{\rho,s}G_1=\{(y,g)\ |\ s(g)=\rho(y)\}$

  \end{itemize}
Using the simplified notation $\alpha(y,g)=yg$, we will say that  the maps $\rho,\alpha$ define a right action of the topological groupoid $(G_0,G_1)$ on the topological space $Y$ when the following conditions are met:
\[\begin{array}{l}
\rho(yg)=t(g),\\
(yh)g=y(hg),\\
yu(\rho(y))=y.
\end{array}\]
\end{defin}

 \section{The space $S(\PP)$ as a fiber bundle on $X$}

Denote, as in the previous sections, the locally compact and locally simply connected space $G_0$ with $X$.
Let $G$ be a metrizable topological group with bounded invariant metric,
$p:P\to X$ a principal (right) $G$-bundle and $\GG=\GG(P)$ its related locally trivial groupoid.

Recall that for any fixed $x_0\in X$, the map $t:s^{-1}(x_0)\to X$ can be identified with $p$ and regarded as a left $G$-bundle $P$.

Let us consider the locally compact space $S(\PP)$ associated with the $C^\ast$-algebra $R_\PP$.
There exists a natural inclusion $\rho_\PP^\ast:C_0(X)\hookrightarrow R_{\PP}$
well defined by restriction of $\eta\mapsto \eta\circ t$ on $P$ by Lemma \ref{lm:RG}.

 Since $X$ is locally compact, by Gelfand construction, $X$ is the space associated to its $C^\ast$-algebra $C_0(X)$.
 It follows that the inclusion $\rho_\PP^\ast:C_0(X)\hookrightarrow R_{\PP}$ is the dual of a continuous surjective map
\begin{equation}\label{eq:rhoP} \rho_\PP: S(\PP)\to X.\end{equation}

\begin{thm}\label{thm:properSP} $\rho_\PP$ is a proper map and it has all the fibers isomorphic to $S(G)$.
Moreover there exists an open covering of locally contractible sets and with compact closure
$\{U_i\}$ of $X$  such that if ${g_{i,j}(x)}$ are the transition functions of the left $G$-bundle $P$ associated to such a covering,
then the following holds:

\begin{enumerate}
\item $\rho_\PP^{-1}(U_i)\cong U_i\times S(G)$ and the left action of $G$ on  $\rho_\PP^{-1}(U_i)$ is given by the left action of $G$ on $S(G)$.
\item The transition functions  $\theta_{ij}:\rho_\PP^{-1}(U_i)|_{U_i\cap U_j}\to \rho_\PP^{-1}(U_j)|_{U_i\cap U_j}$ are defined by $\tau\mapsto \tau g_{i,j}(x)$, i.e. they are induced by transition functions of $P$.
\end{enumerate} \end{thm}
\begin{proof}
Let us consider first that $P|_{U_i}$ is dense in $\rho_\PP^{-1}(U_i)$, then $C_0(\rho_\PP^{-1}(U_i))$ is isomorphic to the restriction of $R_\PP$ to $P|_{U_i}$.
 If $\{U_i\}$ is an open covering of $X$ with $U_i$ locally contractible and with compact closure, then, possibly resizing the open sets $U_i$, $C_0(X)$, by a restriction of its domain to $U_i$, becomes $C(U_i)$. Moreover if the open sets $U_i$ are locally contractible, then the bundle $P|_{U_i}=t^{-1}(U_i)$ has a section on $U_i$. Let such a section be $x\mapsto x'$,
 it follows that there exists the homeomorphism $P|_{U_i}\cong  U_i\times G$ defined as in Proposition \ref{prop:RPloc}.
 We calculate left action of $G=G[x_0]$ on $P|_{U_i}$ as follows. Let identify $g\in G$ with $[u_0g,u_0]\in G[x_0]$,
 acting as in the definition of the isomorphism $\phi_{x_0}$ introduced in Proposition \ref{prop:GP}, then we have
$$g[u_0,v]=[u_0g,u_0][u_0,v]=[u_0g,v]=[u_0gh,x']\mapsto (x,gh)\in U_i\times G.$$ Therefore we can regard such an action as the left action of $G$ on $U_i\times G$.
Proposition \ref{prop:RPloc} and Definition \ref{def:rucbGP} then show that functions $f$ su $P|_{U_i}$, restrictions of functions in $R_\PP$, can be described as continuous functions $F(x,h)$ such that for any $\epsilon>0$ there exists a neighborhood $\VV\ni 1\in G$ such that $|F(gh,x)-F(h,x)|<\epsilon$ for any $x\in U_i$, $k\in G$ and any $g\in\VV$.
 Hence the algebra of such functions coincides with $C(U_i\times S(G))$ which in turns implies that $P|_{U_i}$ is homeomorphic to $U_i\times S(G)$.
The action of $G$ on the functions $f(h,x)$ obtained from $R_\PP$ can be calculated through the formula $F^g(x,h)=F(g^{-1}\cdot(x,h))=F(x,g^{-1}h)$, and this action corresponds to the left action of $G$ on $U_i\times S(G)$ induced by the natural left action of $G$ on $S(G)=\Omega(R_G)$. Then $\rho_\PP^{-1}(U_i)\cong U_i\times S(G)$, with left action of $G$ induced by the canonical one defined on $S(G)$.
Let now be $P|_{U_i}\cong U_i\times G$ e $P|_{U_j}\cong U_j\times G$ two trivializations related by a transition isomorphism
$(x,h)\mapsto (x,hg_{i,j}(x)^{-1})$ on $U_i\cap U_j$.

This isomorphism arises from the relation $x'=x''g_{i,j}(x)$, for $x'$ section on $U_i$ and $x''$ section on $U_j$, therefore by the composition of isomorphisms in the following way
$$(x,h)\mapsto [u_0h,x']=[u_0h,x''g_{i,j}(x)]=[u_0hg_{i,j}(x)^{-1},x'']\mapsto (x,hg_{i,j}(x)^{-1}).$$

As usual, the action on the functions $F(x,h)$ restricted to $R_\PP$ for $x\in U_i\cap U_j$, is $F(x,h)\mapsto F(x,hg_{i,j}(x))$, from $R_\PP$ to $R_\PP$,
and then it induces an automorphism of $(U_i\cap U_j)\times S(G)$. This is enough to show that there exists a gluing map between
$\rho_\PP^{-1}(U_i)$ and $\rho_\PP^{-1}(U_j)$ along  $\rho_\PP^{-1}(U_i\cap U_j)$ which extends the given one for $P|_{U_i}$ and $P|_{U_j}$.
\end{proof}
\section{The right action of $\GG$ on $S(\PP)$}
\begin{thm}\label{thm:SPaction} There exists a continuous action of groupoid
$$\alpha_P:S(\PP)\times_{\rho_\PP,s}G_1\to S(\PP)$$ with anchor map $\rho_\PP$.\end{thm}
\begin{proof}
There exists a natural right action of groupoid
$$ \alpha_P:P\times_{t,s}G_1\to P$$ defined by $\alpha_P([u_0,v], [v,w])=[u_0,w]$.
This action is continuous for the topology on $G_1$ showed in Section \ref{sec:ehresmann} and for that one on $P=s^{-1}(x_0)$, inherited from $G_1$.
Let now consider an open covering $\{U_i\}$ of contractible of $X$ exactly as in Theorem \ref{thm:properSP}, consequently $P|_{U_i}\cong U_i\times G$ e $\rho_\PP^{-1}(U_j)\cong U_i\times G$.
Let $G_{i,j}=\psi(U_i\times U_j\times G)$ an open covering of $G_1$, where $\psi$ is the map defined by $\psi(x,y,k)=[x'k,y']$, using notation introduced in the proof of Theorem \ref{thm:ehres1}. Then for any $i,j$, the map $\alpha_P$ restricts to a continuous map
$$\alpha_P:P|_{U_i}\times_{t,s} G_{i,j}\to P|_{U_j}.$$

Let us notice that $P|_{U_i}\times_{t,s} G_{i,j}\cong U_i\times U_j\times G\times G$ by the map $$\tilde{\psi}:  U_i\times U_j\times G\times G\to P|_{U_i}\times_{t,s} G_{i,j}$$ defined by $\tilde{\psi}(x,y, h, k)=([u_0h,x'],[x'k,y'])$.
Conversely
$U_j\times G\cong P|_{U_j}$ using homeomorphism $(y,k)\mapsto [u_0k,y']$. Using the above identifications the map $\alpha_P$ can be regarded as the map
$${\alpha_P}:(x,y,h,k)\mapsto (y,hk).$$

We consider the pull-back of functions $f\in R_\PP$ through $\alpha_P$.
Recall that the functions $f$ on $P|_{U_j}$ related to $R_\PP$, in the identification $P|_{U_j}\cong U_j\times G$, are continuous functions $F(y,k)$ such that for any $\epsilon>0$ there exists a neighborhood $\VV\ni 1\in G$ such that $|F(y,gk)-F(y,k)|<\epsilon$ for any $k,y$ and any $g\in\VV$.
Then the function $F(x,y, h,k)=f(\alpha_P(x,y,h,k))=F(y,hk)$ satisfies property $|F(x,y,gh,k)-F(x,y,h,k)|<\epsilon$ for any $x,y,h,k$ and any $g\in \VV$.
Consequently the algebra $\alpha_P^\ast(R_\PP)$ restricted to $P|_{U_i}\times G_{i,j}$ is contained in $C_0(U_i\times S(G)\times G_{i,j})$, and therefore in $C_0(S(\PP)|_{U_i}\times G_{i,j})$. Hence the action $\alpha_P$ extends to the continuous function
\begin{equation}\label{eq:alphaPloc}\alpha_P: S(\PP)|_{U_i}\times G_{i,j}\to S(\PP)|_{U_j}\end{equation} for any $i,j$.
Since those actions restricted to $P\times G_1$ glue each other to form the given action $\alpha_P:P\times G_1\to P$, and by density of $P|_{U_i}$ in $S(\PP)|_{U_i}$ for any $i$, then the functions (\ref{eq:alphaPloc}) stick together to form a global action
$\alpha_P:S(\PP)\times G_1\to S(\PP)$.
The required properties on the anchor map $\rho_\PP$ and on the groupoid action are easy to proof, using for example the density of $P$ in $S(\PP)$
and the unique extension of the same properties of the action of $G_1$ on $P$.

\end{proof}
\section{Universal Property of $S(\PP)$}\label{UnivProp}

 Let $\beta: Y\times_{\rho,s} G_1\to Y$ be a right action of a transitive topological groupoid $\GG=(G_0,G_1)$ on a space $Y$ with proper anchor map $\rho:Y\to G_0=X$.

 \begin{thm}[Universal Property of $S(\PP)$]\label{thm:univSP} For any action $$\beta: Y\times_{\rho,s} G_1\to Y$$ with proper anchor map $\rho:Y\to X$  and for any fixed $y\in \rho^{-1}(x_0)$ there exists a unique continuous map $l_{y}:S(\PP)\to Y$ such that $l_{y}(u_0)=y$ and that is compatible with the actions $\alpha_P$ and $\beta$.  \end{thm}

 \begin{proof} Our claim is to prove that the homeomorphism of algebras $l^\ast_{y}:C_0(Y)\to C_0(P)$, defined by $l^\ast_{y}f([u_0,v])=f(\beta(y,[u_0,v]))$, has image in $R_\PP$.
 We will use multiplicative notation
 $\beta(y[u_0,v])=y[u_0,v]$ for the action on $Y$. Let $y_v=y[u_0,v]$ and, since $r^\ast_{y}f([u_0,v])=f(y_v)$, we can assert that if the support of $f$ is in $\rho^{-1}(K)$, with $K\subset X$, compact, then, since $\rho(y_v)=p(v)$, the support of $r^\ast_y(f)$ is in $t^{-1}(K)$.  Hence condition (1) of the Definition \ref{def:rucbGP} is verified.

 For an open set $U\subset X$ with local section $x\mapsto x'$ assume
 $$F(x,h)=l_y^\ast f([u_0h,x'])=l_y^\ast f([u_0,x'h^{-1}])=f(y_{x'h^{-1}}).$$
 Since $y_{x'h^{-1}}$ runs on a compact $\rho^{-1}(K)$, the function $F(x,h)$ turns out to be uniformly continuous on $K\times G$, then for any $\epsilon<0$ there exists $\VV$, a neighborhood of $1\in G$ such that for any $x\in K$, any $h\in G$ and any $g\in \VV$ we have
$|F(x,gh)-F(x,h)|=|f(y_{x'h^{-1}g^{-1}})-f(y_{x'h^{-1}})|<\epsilon$.
Notice that Proposition \ref{prop:RPloc} implies condition (2) of the Definition \ref{def:rucbGP}.
Since $l^\ast_y:C_0(Y)\to R_\PP$ is an homeomorphism of $C^\ast$-algebras, then there exists the desired map $l_y:S(\PP)\to Y$.
Moreover, this map is univocally determined by its restriction to $P$, that is the map $[u_0,v]\to y[u_0,v]$,
which shows the uniqueness of $l_y$ by density of $P$ in $S(\PP)$.
Finally, using again the density of $P$ and the compatibility of the restriction of $l_y$ to $P$, we deduce the compatibility of $l_y$ with the given actions on $S(\PP)$ and on $Y$.
 \end{proof}
  \section{Uniqueness up to isomorphisms of the minimal flows of $S(\PP)$}\label{MinFlow}

  In this section we will prove a similar result to the existence of a universal minimal compact $G$-space showed in \cite{usp}. For reader's convenience we briefly recall some notations and results related to topological groups, which in the present paper, in some extent, we generalize to topological groupoids.
Let $G$ be a topological group, a $G$-space is a topological space $X$ with a continuous action of $G$. A $G$-space $X$ is minimal if the orbit $Gx$ is dense in $X$. The universal minimal compact $G$-space $M_G$ is characterized by the following property: $M_G$ is a minimal compact $G$-space, and for every compact minimal $G$-space $X$ there exists a $G$-map of $M_G$ onto $X$. It is well known that any two universal minimal compact $G$ spaces are isomorphic, see for example \cite{Aus} or \cite{usp}.

Here we show that minimal $\GG$  subflows of $S(\PP)$, on some respects, behave in the same way. Our proof is a slight modification of the one contained in \cite{usp}.
Let us consider $P=s^{-1}(x_0)$ with related anchor map $\rho_P :S(\PP)\rightarrow X $, which extends the map $t:P\to X$.
By Theorem \ref{thm:SPaction} for any $y\in\rho_{\PP}^{-1}(x_0)$  there exists a unique map $l_y:S(\PP)\rightarrow S(\PP)$ such that $l_y(u_0)=y$.
If $y\not=y'$ then $l_y\not=l_{y'}$, moreover, by Theorem \ref{thm:properSP}, $\rho_P^{-1}(x_0)\cong S(G)$ and by such isomorphism $u_0\in p^{-1}(x_0)$, identified with $[u_0,u_0]\in t^{-1}(x_0)\subset \rho_P^{-1}(x_0)$, corresponds to $1\in G$.
We now can define an action $\Phi: S(G)\times S(\PP) \rightarrow S(\PP)$ such that $\Phi (y,z)=l_y(z)$.
Notice that, by universal property of $S(G)$, for $y,z$ in $S(G)=\rho_\PP^{-1}(x_0)$, the product $l_y(z)=yz$ is the same of the one giving to $S(G)$ a semigroup structure and defined in Theorem 2.1 of \cite{usp}. However we warn the reader that in \cite{usp} the product in $S(G)$  is defined as $yz=r_z(y)$, that is, by means of similar maps $r_z$ as the $l_y$ defined here, such that $r_z$ is invariant for the {\em left} action of $G$ on $S(G)$, whereas we have a right $G_1$ invariance (in particular right $G$ invariance) for the maps $l_y$.
\begin{notation} For $y,z\in \rho_\PP^{-1}(x_0)\cong S(G)$, we denote $l_y(z)=yz$.\end{notation}
\begin{prop} For $y,z\in \rho_\PP^{-1}(x_0)$ we have $l_{y}l_{z}=l_{yz}$.
\end{prop}
\begin{proof}
The maps $l_{y}l_{z}$ and $l_{yz}$ are $G_1$-maps that send $u_0$ to the same element $yz$. By uniqueness the thesis follows.\end{proof}
\begin{prop}\label{traslazione}
Let $f:S(\PP)\rightarrow S(\PP)$ be a $G_1$-map (in particular a $G$-map) then $f=l_y$ for some $y\in\rho_{\PP}^{-1}(x_0)$
\end{prop}
\begin{proof}
Denote $f(u_0)=y$ and consider the map $l_y$. Then $f$ and $l_y$ send $u_0$ to the same point and by uniqueness of $l_y$ we have $f=l_y$.
\end{proof}

\begin{prop}\label{idempotente}
Let $M(\PP)$ be a $G_1$-minimal flow of $S(\PP)$ then $M(\PP)=r_y(S(\PP))$ for some idempotent $y\in\rho_{\PP}^{-1}(x_0)\cap M(\PP)$.
Moreover $r_y$ is the identity on $M(\PP)$.
\end{prop}

\begin{proof} First observe that $M(\PP)\cap \rho_{\PP}^{-1}(x_0)$ is necessarily a $G$-minimal subflow of $\rho_{\PP}^{-1}(x_0)=S(G)$, otherwise $(M(\PP)\cap \rho_{\PP}^{-1}(x_0))G_1$ would be a proper subflow of $M(\PP)$, against the minimality of $M(\PP)$.
 By the results of Section 3 of \cite{usp} there exists a idempotent $y\in M(\PP)\cap \rho_{\PP}^{-1}(x_0)$ and one can consider the map $l_y:S(\PP)\to S(\PP)$, which, by minimality, is such that $l_y(M(\PP))=M(\PP)$. More precisely, from $y^2=y$ one finds $yyz=yz$ for any $z\in M(\PP)$ and, since the $yz$ span $M(\PP)$, the map $l_y$ is the identity on $M(\PP)$. \end{proof}
\begin{prop}\label{biezione}
Every $\GG$-map $f:M(\PP)\rightarrow M(\PP)$ is bijective.
\end{prop}
\begin{proof}
 Composing $f$ with $l_y:S(\PP)\to M(\PP)$ of the Proposition \ref{idempotente}, and using Proposition \ref{traslazione}, we obtain
$f\circ l_y=l_z$ for some $z\in M(\PP)$. As $l_y|_{M(\PP)}$ is the identity map of $M(\PP)$, we obtain $f=l_z|_{M(\PP)}$. The rest of the proof runs exactly as the proof of Proposition 3.4 of \cite{usp}. \end{proof}
\begin{prop}\label{unicità}
$M(\PP)$ is unique up to isomorphisms.
\end{prop}
\begin{proof}
Let $M$ and $M'$ be two minimal flows of $S(\PP)$, according to universality of $S(\PP)$ there exists a $G_1$-map $f$ which sends $M$ in $M'$, reversing roles we obtain a similar map from $M'$ to $M$. By minimality of $M$ and $M'$ they are both onto functions, moreover their composition is $G_1$-map from $M$ to $M$, and, according to Proposition \ref{biezione} is actually a bijection, hence $f$ is injective, therefore $M$ and $M'$ are isomorphic.
\end{proof}
\begin{oss} Of course any minimal $G_1$ subflow $M(\PP)$ of $S(\PP)$ is a {\em universal minimal flow} for right $\GG$ actions with proper anchor map. Indeed if $M$ is a minimal flow, consider a $\GG$-invariant map $f:S(\PP)\to M$, whose existence is guaranteed by the universal property of $S(\PP)$, then $f(M(\PP))=M$ by minimality.  \end{oss}
\section{The case of extremely amenable groups}
It turns out that most well known groups, e.g. discrete groups, lie groups, or in general locally compact groups, have very big compactifications $S(G)$ and also very big universal minimal flows $M_G$. For example when $G$ is discrete then $S(G)=M_G=\beta G$, the Stone-Cech compactification of $G$, a fact originarily proved by Ellis and also a consequence of Veech's theorem \cite{Vee}. On the other hand some very huge groups, like $\operatorname{Homeo}_+(S^1)$, the group of orientation preserving self-homeomorphism of $S^1$, with the compact open topology, or $U(H)$, the unitary group of a infinite dimensional separable Hilbert space, with the strong operator topology, have small $M_G$, i.e. $M_G=S^1$ for the first and $M_G=\{*\}$ is a singleton in the latter case. We refer to section 4 of \cite{pestov} for a detailed discussion these and many other examples. In particular $U(H)$ is a prototypal example of a {\em extremely amenable group}, a fact proved in \cite{gromi}. We recall the general definition of this class of groups.
\begin{defin} A topological group $G$ is said to be {\em extremely amenable} if any continuous action of $G$ on a compact Hausdorff space has a fixed point. This is equivalent to say that $M_G$ is a point. \end{defin}
Assume that $\GG$ is a transitive groupoid with $G[x]=G$ an extremely amenable group. Then one has the following result about the actions of $\GG$ with proper anchor map, of which we give a proof independent of the rest of the paper for convenience of the reader.
\begin{thm}\label{thm:stupid} Let $G$ be a topological group. Then $G$ is extremely amenable if and only if for any locally trivial transitive groupoid $\GG$ with  structural group $G$ and any action of $\GG$ on a topological space $X$ with proper anchor map $\rho:X\to G_0$  there exists a continuous $\GG$-invariant section  $G_0\to X$ of $\rho$. \end{thm}
\begin{proof} ($\Rightarrow$) Let $\alpha:X\times_{\rho,s} G_1\to X$ be an action di $\GG$ on $X$ with $\rho:X\to G_0$ a proper map. Then for any $x\in G_0$ there exists an induced action $G[x]=G$ on $\rho^{-1}(x)$, which is a compact space.

Since $G$ is extremely amenable there exists a fixed point $z\in \rho^{-1}(x)$. Let us consider the orbit $zG_1\subset X$ and let us show that it is the image of a continuous section $\sigma:G_0\to X$. The fact the $zG_1$ is the image of a set-theoretical section is a consequence of the property that for any $x'\in G_0 $ one has $|\rho^{-1}(x')\cap zG_1|=1$.  By the transitivity of $\GG$  there exists $g\in G_1$ such that $s(g)=x$ and $t{g}=x'$. Then $t(zg)=x'$, therefore $|\rho^{-1}(x')\cap zG_1|\geq 1$. If $z',z''\in \rho^{-1}(x')\cap zG_1$, then $z'=zg$ e $z''=zh$ for suitable $g,h\in G_1$ and one has $t(g)=t(h)=\rho(z')=\rho(z'')=x'$. Hence $gh^{-1}\in G[x]$ and one has $z=zgh^{-1}$, which implies $z'=zg=zh=z''$.

Now let us show that $\sigma$ is continuous. As $\sigma $ is a bijection between $G_0$ and Im$(\sigma)=zG_1$, with inverse equal to the restriction $\rho|_{zG_1}$, then it is sufficient to show that $\rho|_{zG_1}$ is a open map. Let $W\cap zG_1$ an open set in the induced topology on $zG_1$, with $W$ open in $X$. If $z'\in zG_1\cap W$ then $z'=zh$ for ssome $h$ and $\rho(z')=t(h)$. It follows that $\rho(zG_1\cap W)=\{t(h)\ |\  zh\in W\}$. The set $\{(z,h)\ |\ zh\in W\}$ is equal to $\alpha^{-1}(W)\cap (z\times s^{-1}(x))$, which is homeomorphic to a open $W_0\subset s^{-1}(x)$, since $\alpha^{-1}(W)\subset X\times_{\rho,s}G_1$ is open by the continuity of the action $\alpha$. Now let us consider the restriction of the target map $t:s^{-1}(x)\to G_0$.  It easy to see that $\GG$ \`e locally trivial implies that $t$ is open.
Then we have $\rho(zG_1\cap W)=\{t(h)\ |\  zh\in W\}=t(W_0)$ open.
\paragraph{$(\Leftarrow)$} Let $X\times G\to X$ an action $G$ on a compact $X$. We consider the group as a groupoid, i.e. $G_1=G$, $G_0=\{1\}$ and the group action as a groupoid action $X\times_{\rho,s} G_1\to X$, with trivial anchor map $\rho:X\to\{1\}$. Then $\rho$ is proper and an invariant section corresponds to a fixed point $x=\sigma(1)\in X$ for the action of $G$.
\end{proof}

\begin{oss} The theorem above applies in particular to the action of $\GG(P)$ on $S(\PP)$ studied in the preceding sections, when $P$ is a principal bundle with structural group $G$ extremely amenable. In particular we obtain the result that any minimal flow $M(\PP)$ of $S(\PP)$ is the image of a section $\sigma : X=G_0\to S(\PP)$, which is also $\GG$ invariant. In particular $S(\PP)\setminus P$ is a union of images of invariant sections.\end{oss}
Recall that if a principal bundle $P\to X$ has a section $\sigma: X\to P$ then $P\to X$ is trivial. In the situation above we have invariant sections {\em at infinity}, i.e. with images contained in $S(\PP)\setminus \PP$. The following question then arises naturally
\begin{domanda}\label{conj} Let $P\to X$ a principal bundle with $G$ extremely amenable, arcwise connected and compactly generated. Is then $P$ a trivial bundle?\end{domanda} \begin{oss}
The condition on the topology of $G$ to be compactly generated is a natural assumption if one wants to apply results from algebraic topology. Indeed the same definition of homotopy groups $\pi_i(G)$ involves maps from the compact spaces $S^i$ to $G$ and therefore it is natural to assume that maps from compact spaces to $G$ detect the topology of $G$.

Note that if the question above has a positive answer, then $G$ is {\em aspherical} or {\em quasi-contractible}, that is all homotopy groups $\pi_i(G)$ are trivial. Indeed, from the theory of fiber bundles,  it is well known that the set Prin$_G(S^n)$ of isomorphism classes of principal $G$-bundles with base $X=S^n$ is in bijective correspondence with $\pi_{n-1}(G)$, for any $n\geq 1$.

We first recall that both the hypotheses of Question \ref{conj} and the conclusion are true, as far we know, for at least three groups. Observe that in all those cases the group $G$ is metrisable (as in the present paper), hence compactly generated.
Indeed,
 \begin{itemize}
   \item $U(H)$, the unitary group of infinite dimensional separable Hilbert space, is extremely amenable \cite{gromi}, on the other hand by Kuiper's theorem \cite{kui} is aspherical;
   \item $Iso(U)$, the isometry group of the Urysohn space, is extremely amenable \cite{GioPe} \cite{GioPe1}, on the other hand it is isomorphic to the infinite dimensional separable Hilbert space hence it is aspherical \cite{mel};
   \item $Aut(I,\mu)$, the group of measure preserving automorphisms of $I$, the unit interval, is extremely amenable \cite{GioPe} \cite{GioPe1}, on the other hand is isomorphic to the infinite dimensional separable Hilbert space hence it is aspherical \cite{ngu}.
 \end{itemize}
 For the metrisability of $U(H)$ with the strong operator topology one can see for example \cite{EspUr}.

    \end{oss}
\begin{oss} In case Question \ref{conj} admits a positive answer, at least under suitable topological conditions on $G$, a proof of this along the line of the present paper would require to prove the following.
\paragraph{\em Claim} Let $G$ be a topological group with the same topological assumptions as in Question \ref{conj} (but not necessarily extremely amenable). Let $P\to X$ be a non trivial principal bundle with structural group $G$, and let $\GG$ be its associated groupoid. Then there exists some $\GG$ equivariant compactification of $\overline{P}\to X$ of $P$ such that the fiber bundle $\overline{P}\setminus P\to X$ does not admit ($\GG$-equivariant) global sections.
\end{oss}
For some easy finite-dimensional topological group, like $G=\CC^\ast$, the claim above is true, and can be proved by elementary algebraic topology constructions, but for infinite dimensional groups like $G=U(H)$ it is far from clear to us if the claim is true.

\end{document}